\documentclass[11pt,a4paper]{article}
\usepackage{hyperref}
\usepackage[english]{babel}
\usepackage[utf8]{inputenc}
\usepackage[T1]{fontenc}
\usepackage{amsmath,amsthm,amssymb,amsfonts}
\numberwithin{equation}{section}
\usepackage{enumerate}
\usepackage{faktor}
\usepackage{dsfont}
\usepackage{scalerel,stackengine}
\usepackage{textcomp}
\usepackage{graphicx}
\usepackage{extarrows}
\usepackage{epigraph}
\usepackage{graphicx}
\usepackage{subcaption}
\usepackage{sidecap}
\usepackage{indentfirst}
\usepackage{tikz}
\usepackage{tikz-cd}
\usetikzlibrary{shapes.misc,arrows,matrix,patterns,decorations.markings,positioning}
\tikzset{cross/.style={cross out, draw=black, minimum size=2*(#1-\pgflinewidth), inner sep=0pt, outer sep=0pt},
cross/.default={4.5pt}}
\usepackage{pgfplots}
\usepackage{caption}
\usepackage{float}
\usepackage{color}
\usepackage{epstopdf}
\usepackage{epsfig}
\usepackage{stmaryrd}
\usepackage{color}
\usepackage{geometry}
\usepackage{multicol}
\usepackage{verbatim}

\DeclareMathOperator{\tb}{tb}
\DeclareMathOperator{\rot}{rot}
\DeclareMathOperator{\self}{sl}
\DeclareMathOperator{\lk}{\ell k}
\DeclareMathOperator{\writhe}{wr}

\renewcommand{\geq}{\geqslant}
\renewcommand{\leq}{\leqslant} 
\renewcommand{\epsilon}{\varepsilon}
\newcommand{\R}{\mathbb{R}}

\newcommand{\Z}{\mathbb{Z}}
\newcommand{\Q}{\mathbb{Q}}

\newcommand{\C}{\mathbb C}
\newcommand{\dd}{\mathrm{d}}

\newcommand{\norm}[1]{\left\Vert#1\right\Vert}

\newcommand{\xist}{\xi_{\text{st}}}

\DeclareFontFamily{U}{mathx}{\hyphenchar\font45}
\DeclareFontShape{U}{mathx}{m}{n}{
      <5> <6> <7> <8> <9> <10>
      <10.95> <12> <14.4> <17.28> <20.74> <24.88>
      mathx10
      }{}
\DeclareSymbolFont{mathx}{U}{mathx}{m}{n}
\DeclareFontSubstitution{U}{mathx}{m}{n}
\DeclareMathAccent{\widecheck}{0}{mathx}{"71}
\DeclareMathAccent{\wideparen}{0}{mathx}{"75}

\newtheorem{teo}{Theorem}[section]
\newtheorem*{teo*}{Theorem}
\newtheorem{lemma}[teo]{Lemma}
\newtheorem{prop}[teo]{Proposition}
\newtheorem*{prop*}{Proposition}

\newtheorem{cor}[teo]{Corollary}

\usepackage{xpatch}
\makeatletter
\xpatchcmd{\@thm}{\thm@headpunct{.}}{\thm@headpunct{}}{}{}
\makeatother

\theoremstyle{definition}

\pgfplotsset{compat=1.13}
\begin{document}
\title{On Bennequin type inequalities for links in tight contact 3-manifolds}
\author{\scshape{Alberto Cavallo}\\ \\
 \footnotesize{Alfr\'ed R\'enyi Institute of Mathematics,}\\
 \footnotesize{Budapest 1053, Hungary}\\ \\ \small{cavallo\_alberto@phd.ceu.edu}}
\date{}

\maketitle

\begin{abstract}
 We prove that a version of the Thurston-Bennequin inequality holds for Legendrian and transverse links
 in a rational homology contact 3-sphere $(M,\xi)$, whenever $\xi$ is tight. More specifically, we show that the self-linking 
 number
 of a transverse link $T$ in $(M,\xi)$, such that the boundary of its tubular neighborhood consists of incompressible tori, is bounded by the 
 Thurston norm $\norm{T}_T$ of $T$.
 A similar inequality is given for Legendrian links by using the notions of positive and negative transverse push-off.
 
 We apply this bound to compute the tau-invariant for every strongly quasi-positive link in $S^3$. This is done by proving 
 that
 our inequality is sharp for this family of smooth links. Moreover, we use a stronger Bennequin inequality, for links in the tight 3-sphere,
 to generalize this result to quasi-positive links and determine their maximal self-linking number.
\end{abstract}

\section{Introduction}
The Thurston-Bennequin inequality is a very powerful result in contact topology. One of its most important implications is that, provided the 
structure is tight, the 
Thurston-Bennequin numbers of all the Legendrian knots, with the same smooth knot type, are bounded from above. This property actually
characterizes
tight contact structures on 3-manifolds; in fact, when the structure is overtwisted, it is always possible to increase the Thurston-Bennequin
number indefinitely, without changing the smooth isotopy class of the knot. See \cite{Etnyre} for details.

This paper consists of two parts. In the first one, we prove that an analogous inequality, involving the Thurston norm $\norm{L}_T$ of a link 
\cite{Thurston}, holds for Legendrian and
transverse links in every rational homology contact 3-sphere,
equipped with a tight structure. 

A proof for a version of the Thurston-Bennequin inequality was given by Eliashberg in \cite{Eliashberg}.
For Legendrian links in $(S^3,\xi_{\text{st}})$ an analogous result, and the resulting upper bound for the maximal
Thurston-Bennequin number, was found first by Dasbach and Mangum in \cite{DM}.
Furthermore, if we consider only Legendrian and transverse knots then our bounds coincide with the ones of Baker and Etnyre in \cite{BE}. 

In the second part we use the already existing Bennequin inequalities in the tight 3-sphere \cite{Cavallo,DM} to give some results on quasi-positive 
links in $S^3$, which are the closure of quasi-positive braids. These are braids that are written as
$(w_1\sigma_{j_1}w_1^{-1})\cdot...\cdot(w_b\sigma_{j_b}w_b^{-1})$ for some $b\geq0$ and where the $\sigma_{j_i}$'s are the generators of the braids group,
see \cite{Etnyre,Hedden}.

We say that a component of a link $L$ is compressible if the boundary of
its neighborhood is a compressible torus in the complement of $L$. We denote with $o(L)$ the number of compressible components in $L$ and
with $\mathfrak o(L)=t_1^{-1}+...+t_{o(L)}^{-1}$ the \emph{compressibility term} of $L$, where $t_i$ is the order of the 
homology class represented by the $i$-th compressible component of $L$ in $H_1(M;\Z)$.
\begin{teo}
 \label{teo:main}
 Suppose that $L$ is a Legendrian link in a tight contact 3-manifold $(M,\xi)$ and $M$ is a rational homology sphere. Then we have that
 \[\emph{tb}(L)+|\emph{rot}(L)|\leq\norm{L}_T-\mathfrak o(L)\:.\]
 Furthermore, suppose that $T$ is a transverse link in $(M,\xi)$. Then we have that \[\emph{sl}(T)\leq\norm{T}_T-\mathfrak o(T)\:.\]
\end{teo}
In \cite{Cavallo} we generalize the Ozsv\'ath-Szab\'o $\tau$-invariant for knots in $S^3$ to links. This can be used to improve
the Thurston-Bennequin inequality in $(S^3,\xist)$; in this way we are able to compute to value of $\tau$ for all quasi-positive links.

We recall that a transverse link in $(S^3,\xist)$ has bounded self-linking number. This means 
that we can denote the maximal self-linking number of a link $L$ with $\text{SL}(L)$.
\begin{teo}
 \label{teo:self}
 Suppose that $L$ is an $n$-component quasi-positive link in $S^3$. Then we have that
 \[\emph{SL}(L)=b-d\:\:\:\:\:\text{ and }\:\:\:\:\:\tau(L)=\dfrac{b-d+n}{2}\:,\] where $L$ is the closure of a quasi-positive $d$-braid 
 $B$.
\end{teo}
Later, we introduce the subfamily of connected transverse
$\C$-links: these are links such that the surface $\Sigma_B$, associated to a quasi-positive braid $B$ for $L$ as we describe in 
Subsection \ref{subsection:max_self}, is connected.
For such links we are also able to prove some results on the \emph{slice genus} $g_4$, where for 
a link $L$ in $S^3$ we define $g_4(L)$ as the smallest genus of a connected, compact, oriented surface properly embedded in $D^4$ and
whose boundary is $L$.
\begin{figure}[ht]
 \centering
 \psfig{file=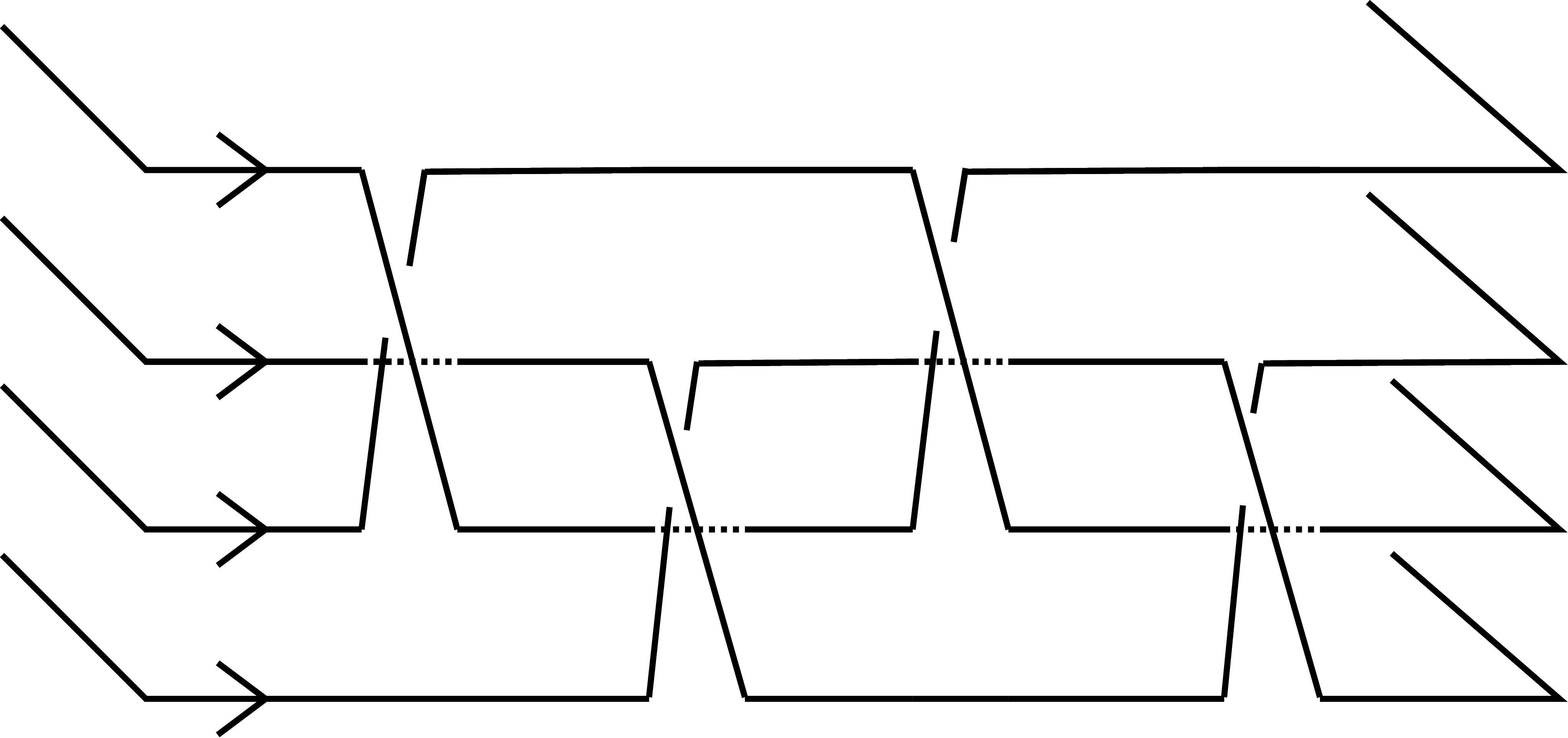,width=12cm}   
 \caption{The quasi-positive surface determined by $d=4$ and $\textbf b=(\sigma_{13},\sigma_{24},\sigma_{13},\sigma_{24})$.}
 \label{QP}
\end{figure}
More specifically, we show that, given two connected transverse $\C$-links $L_1$ and $L_2$, the slice genus of a 
connected sum $L_1\#L_2$ is the sum of the ones of $L_1$ and $L_2$:
\[g_4(L_1\#L_2)=g_4(L_1)+g_4(L_2)\:.\] This result, which follows from \cite{Hedden} in the case of knots, appears to be new for links.
\begin{cor}
 \label{cor:slice}
 The slice genus $g_4$ is additive under connected sums if we restrict to the family of connected transverse $\C$-links. 
\end{cor}
Theorem \ref{teo:main} also allows
us to state some results about strongly quasi-positive links.
Those are links in the 3-sphere that bound a quasi-positive surface, 
a surface constructed
from a collection of $d$ parallel disks by attaching $b$ negative bands. 
An example of a quasi-positive surface is shown in Figure \ref{QP}. 
The proof of the following theorem is similar to the one of Hedden for knots in \cite{Hedden}; moreover, the same result was obtained independently by Hayden in \cite{Hayden} using different techniques. \newpage
\begin{teo}
 \label{teo:strong}
 Consider a strongly quasi-positive link $L$ with $n$ components in $S^3$. 
 Then the Thurston-Bennequin inequality is sharp for $L$, in the sense 
 that there exists
 a Legendrian representative $\mathcal L$ of $L$ such that
 \[\emph{tb}(\mathcal L)+|\emph{rot}(\mathcal L)|=2\tau(L)-n=\norm{L}_T-o(L)=-\chi(F)\:,\] where $o(L)$ is the number of disjoint unknotted unknots 
 in $L$ and $F$ is a quasi-positive surface for $L$. In particular, it is 
 \[\tau(L)=\dfrac{\norm{L}_T-o(L)+n}{2}=\dfrac{b-d+n}{2}\:,\] where $F$ consists of $d$ disks and $b$ negative bands.
\end{teo}
Positive links, that are links which admit a diagram with only positive crossings,
are a special family of strongly quasi-positive links, see \cite{Nakamura,Rudolph}. Hence, using the oriented resolution of an oriented link
diagram, which is a planar subspace obtained from a diagram by resolving all the crossings in the unique way that preserves the orientation,
we obtain a refinement of Theorem \ref{teo:strong}. This follows from a result of K\'alm\'an \cite{Kalman}.  

We denote with $\text{TB}(L)$ the maximal Thurston-Bennequin number of a link $L$. We also write
$g_3(L)$ for the \emph{Seifert genus} of $L$, the smallest genus of a connected, compact, oriented surface embedded in $S^3$ and 
whose boundary is $L$. We 
recall that such surfaces are called \emph{Seifert surfaces}. Moreover, a link $L$ can always be written as the disjoint union of its
split components, that we define later in Subsection \ref{subsection:split}.
\begin{prop}
 \label{prop:positive}
 Suppose $L$ is an $n$-component link in $S^3$ with a positive diagram $D$ and split components $L_1,...,L_r$. Then we have that
 \[\emph{TB}(L)=2\tau(L)-n=2g_3(L)+n-2r\:.\]
\end{prop}
This paper is organized as follows. In Section \ref{section:two} we define the Thurston norm $\norm{L}_T$ of a link $L$ in a rational 
homology 3-sphere. In Section \ref{section:three} we define the classical invariants of Legendrian and transverse links in rational homology
spheres and then we prove the Thurston-Bennequin inequalities in our main theorem. Finally, in Section 
\ref{section:four} we apply these inequalities to quasi-positive links in $S^3$.
Such applications include Theorems \ref{teo:self} and \ref{teo:strong} and Proposition \ref{prop:positive}.

\paragraph*{Acknowledgements:}
The author would like to thank Andr\'as Stipsicz for his advices and Irena Matkovi\v{c} for all the help she gave me.
The author also wants to thank John Etnyre and Tam\'as K\'alm\'an for their useful observations; and the referee  for his suggestions.
The author is supported by a Young Research Fellowship from the Alfr\'ed R\'enyi Institute of Mathematics and a Full Tuition Waiver
for a Doctoral program at Central European University.

\section{The Thurston norm of a link}
\label{section:two}
\subsection{Definition}
\label{subsection:norm}
Thurston in \cite{Thurston} introduced a semi-norm on the homology of some 3-manifolds. In this section
we recall the construction in
the specific case of rational homology spheres.

Let us consider a compact, connected, oriented 3-manifold $Y$ such that its boundary consists of some tori. We call the 
\emph{complexity} of a compact, oriented surface $F$ properly embedded in $(Y,\partial Y)$,
the integer \[\chi_-(F)=-\sum_{i=1}^m\chi(F_i)\:,\]
where $F_1,...,F_m$ are the connected components of $F$ which are not closed and not diffeomorphic to disks.
If $m$ happens to be zero then we say that the complexity is also zero.
Then we define the function \newpage \[\norm{\cdot}_Y:H_2(Y,\partial Y;\Z)\longrightarrow \Z_{\geq 0}\]
by requiring that \[\norm{a}_Y=\min\left\{\chi_-(F)\right\}\:,\] where $F$ is a surface as above that represents the
relative homology class $a\in H_2(Y,\partial Y;\Z)$. 

It can be shown \cite{Calegari} that 
\begin{enumerate}
 \item $\norm{la}_Y=|l|\cdot\norm{a}_Y$ for any $l\in\Z$ and $a\in H_2(Y,\partial Y;\Z)$;
 \item $\norm{la+mb}_Y\leq|l|\cdot\norm{a}_Y+|m|\cdot\norm{b}_Y$ for any $l,m\in\Z$ and $a,b\in H_2(Y,\partial Y;\Z)$.
\end{enumerate}
Suppose from now on that $M$ is a rational homology 3-sphere and $L$ is a smooth link in $M$. This implies that 
$H_2(M_L,\partial M_L;\Q)\cong\Q^n$, where $M_L=M\setminus\mathring{\nu(L)}$ is the complement of the open tubular neighborhood and $n$ the number of components of $L$.
Then the Thurston semi-norm \[\norm{\cdot}_{M_L}:H_2(M_L,\partial M_L;\Q)\longrightarrow\Q_{\geq 0}\] is defined as before for
integral homology classes and extended to the whole $\Q^n$ by saying that $\norm{la}_{M_L}=|l|\cdot\norm{a}_{M_L}$ for every
$l\in\Q$ and $a\in H_2(M_L,\partial M_L;\Q)$.

If $L$ is a null-homologous $n$-component link in $M$ then we can easily extract a number from $\norm{\cdot}_{M_L}$. 
Namely, we define the Thurston norm of $L$ as the integer \[\norm{L}_T=\norm{[F]}_{M_L}\:,\] where $F$ is a Seifert surface for 
$L$. We recall that a Seifert surface is a connected, compact, oriented surface $F$ which is embedded in $M$ and it is such 
that $\partial F=L$. It is clear that, in a rational homology sphere, all the Seifert surfaces of a link $L$ represent the
same relative homology class in $M_L$ and then the Thurston norm $\norm{L}_T$ is well-defined.

In \cite{OS} Ozsv\'ath and Szab\'o prove that the link Floer homology group $\widehat{HFL}(L)$ of a link in $S^3$ detects
the Thurston norm. This result has been generalized to null-homologous links in rational homology spheres by Ni \cite{Ni}.
\begin{teo}
 \label{teo:detect}
 Suppose that $L$ is a null-homologous $n$-component link in a rational homology 3-sphere $M$. Then we have that
 \[\max\left\{s\in\Q\:|\:\widehat{HFL}_{*,s}(L)\neq\{0\}\right\}=\dfrac{\norm{L}_T-o(L)+n}{2}\:,\] where $o(L)$ denotes the
 number of disjoint unknotted unknots in $L$.
\end{teo}
The definition of the Thurston norm $\norm{\cdot}_T$ does not immediately extend to links that are not null-homologous;
in fact, these links do not admit Seifert surfaces. In order to avoid this problem we need to introduce rational Seifert
surfaces, see also \cite{BE}. Let us consider an $n$-component link $L$ with order $t$ in $M$. This means that $[L]$ has order $t$ in
the group $H_1(M;\Z)$, which is finite because $M$ is a rational homology sphere. Then $F$ is \emph{rationally bounded} by $L$
if there is a map $j:\Sigma\rightarrow M$, where $\Sigma$ is a compact, oriented surface with no closed components, such that 
\begin{itemize}
 \item $j(\Sigma)=F$;
 \item $j\lvert_{\mathring{\Sigma}}$ is an embedding of the interior of $\Sigma$ in $M\setminus L$;
 \item $j\lvert_{\partial\Sigma}:\partial\Sigma\rightarrow L$ is a $t$-fold cover of all the components of $L$.
\end{itemize}
Moreover, if $F$ is also connected then we call it a \emph{rational Seifert surface} for $L$. 

Let us denote with $\{\mu_1,...,\mu_n\}$ the meridian curves of $L$; each $\mu_i$ is embedded in $\partial\nu(L_i)$, where $L_i$ is the $i$-th
component of $L$. By Alexander duality we have that two properly embedded surfaces $F_1$ and $F_2$ in $M_L$ represent the same relative
homology class if and only if $F_1\cdot\mu_i=F_2\cdot\mu_i$ for every $i=1,...,n$, where here we mean the algebraic intersection of the surface
$F_j$ with the curve $\mu_i$ in the 3-manifold $M_L$. \newpage
\begin{lemma}
 \label{lemma:class}
 Suppose that $F$ is a compact, oriented surface properly embedded in $M_L=M\setminus\mathring{\nu(L)}$ with $F\cdot\mu_i=t$ for every $i=1,...,n$,
 where $M$ is a rational homology sphere
 and $L\hookrightarrow M$ is an $n$-component link with order $t$. 
 
 Then there exists an $F'$ in $M$ which is rationally bounded by $L$ and it is such that $\chi_-(F'\cap M_L)\leq\chi_-(F)$ and $F'\cap M_L$ represents 
 the same relative homology class of $F$.
\end{lemma}
\begin{proof}
 First, we observe that trivial properly embedded disks in $M_L$, which are connected components of $F$, do not increase the complexity 
 $\chi_-(F)$; then we can just delete them. Moreover, suppose
 that there are other boundary components of $F$, whose algebraic intersection with $\mu_i$ 
 is zero for every $i=1,...,n$. Then those components are homologically trivial in the tori $\partial\nu(L)$, which means that they are circles
 which bound disks. We can push these disks slightly out of $\partial\nu(L)$, starting from the innermost ones, and then cap off the surface.
 After removing all the closed components we obtain a new surface whose complexity is smaller or equal to the one of $F$.
 
 In the second step we show that we can take $F'$ such that $F'\cap\partial\nu(L)$ consists of essential, parallel, simple closed curves all 
 oriented in the same direction. Therefore, suppose there are two components $C_1$ and $C_2$ of $F'\cap\partial\nu(L_i)$ with opposite 
 orientations; then there is an innermost pair of such components, say precisely $C_1$ and $C_2$ without loss of generality,
 such that $C_1\cup C_2$ is the boundary of an annulus $A\subset\partial\nu(L_i)$ with $C_i\cap\mathring A$ empty. We can alter the surface
 $F'$ by attaching a copy of the annulus $A$ to $F'$ and pushing it into the interior of $M_L$. This operation does not change the relative 
 homology class of $F'$ nor its complexity.
 
 We have now obtained that $F'$ is a compact, oriented surface properly embedded in $M_L$,
 representing the same relative homology class of $F$, such that $\chi_-(F')\leq\chi_-(F)$  
 and $F'\cap\partial\nu(L_i)$ is a link with slope $(t,s)$, where $s\in\Z$, in the torus
 $\partial\nu(L_i)$ for every $i=1,...,n$. To conclude we need to extend $F'$ inside $\nu(L_i)$ in a way that it is rationally bounded by $L$. This 
 can always be done and it is proved in \cite{BE}. This completes the proof.
\end{proof}
From Lemma \ref{lemma:class} and \cite{Thurston} we have that rational Seifert surfaces exist for every link in a rational homology sphere.
Hence, we say that the Thurston norm of a link $L$ with order $t$ in $M$ is the rational number
\[\norm{L}_T=\norm{\dfrac{[F]}{t}}_{M_L}\:,\] where $F$ is a rational Seifert surface for $L$. In the same way as before, 
we have that $\norm{L}_T$ is well-defined. Moreover, if $L$ is null-homologous then the two definitions coincide.

\subsection{Split and non-split links}
\label{subsection:split}
We say that a link $L\hookrightarrow M$ is split if $L$ is the disjoint union of $L_1$ with $L_2$ and there is a separating, embedded 2-sphere $S$ 
in $M$ such
that $S$ gives the connected sum decomposition $M=M_1\#_S M_2$ and $L_i\hookrightarrow M_i$ for $i=1,2$. When this happens we write 
$L=L_1\sqcup L_2$.
Otherwise $L$ is called non-split.
Every link $L$ can be written as $L_1\sqcup...\sqcup L_k$ where each $L_i$ is a non-split link in $M$. Moreover, we show that 
the Thurston norm is additive under disjoint unions. 

A component $K$ of a link $L$ in a rational homology sphere
$M$ is called \emph{compressible} if $K$ is rationally bounded by a surface $F$, disjoint from $L\setminus K$, 
that is the image of a map $j:D\hookrightarrow M$, where $D$ is a disk. 
We recall that a surface $S$ is compressible in $Y$ if a non-trivial circle in $S$ bounds a disk in $Y\setminus S$.
Hence, we observe that $K$ is a compressible component of $L$ if and only if 
$L=K\sqcup L'$, where $K\hookrightarrow M_1$, and the boundary of the neighborhood of $K$ is a compressible torus in $M_1\setminus\mathring{\nu(K)}$. 
Moreover, if $K$ is 
null-homologous then this is equivalent to say that $K$ is an unknot disjoint from $L'$.  
\begin{prop}
 \label{prop:split}
 Let us consider a link $L$ in $M$ such that $L=L_1\sqcup L_2$, where each $L_i$ is 
 embedded in $M_i$ with $M=M_1\#M_2$. Then we have that \newpage
 \[\norm{L}_T=\norm{L_1}_T+\norm{L_2}_T\] and the order of $L$ in $M$ is $t=\emph{lcm}(t_1,t_2)$, where $t_i$ is the order of $L_i$ in $M_i$ for 
 $i=1,2$.
\end{prop} 
\begin{proof}
 We suppose first that $L_1$ and $L_2$ have no compressible components. 
 Let us consider a surface $F$ which is rationally bounded by $L$. Fix a separating 2-sphere $S\hookrightarrow M$, that gives a connected sum 
 decomposition of $M$,
 in a way that $S$ intersects $F$ transversely in a collection of circles $\mathcal S$. Moreover, we suppose that there exist neighborhoods for
 $L_1$ and $L_2$ which are 
 disjoint from $S$, each one lying in one component of $M\setminus S$. Take the map $j:\Sigma\hookrightarrow M$ which defines $F$; if the 
 neighborhoods
 are chosen small enough then we can also suppose that they intersect $F$ only in the image of a neighborhood of $\partial\Sigma$.
 
 Now each circle in $\mathcal S$ separates $S$ into two disks. Let $C\subset\mathcal S$ be a circle that is innermost on $F$. This means that
 $C$ bounds a disk $D$ in $S$, the interior of which misses $F$. Now use $D$ to do surgery on $F$ in the following way: create a new surface 
 $\widehat F$ from $F$
 by deleting a small annular neighborhood of $C$ and replacing it by two disks, each a ``parallel'' copy of $D$, one on either side of $D$.
 We then perform the surgery on $F$ described before on all the circles in $\mathcal S$.
 At this point $\widehat F$ may have closed components, but in this case we just delete them. Therefore, we are left
 with a disconnected surface whose connected components, that are no longer closed, stay in $M_i$ according wether they bound a component of $L_i$.
 We call $F_1$ the surface given by the union of the components of $\widehat F$ of the first type and $F_2$ the other one. 
 
 We clearly have that $F_1$ and $F_2$ are
 disjoint; moreover, they lie in $M_1$ and $M_2$ respectively, they do not intersect $S$ and are such that
 \begin{itemize}
  \item they have no closed components;
  \item $M_L\cap F_i$ is a proper submanifold of $M_L$ with no disk components, 
        while $F_i$ coincide with $F$ in a small neighborhood of $L$ for $i=1,2$;
  \item $\chi(F_1)+\chi(F_2)=\chi(F_1\sqcup F_2)\geq\chi(F)$.
 \end{itemize}
 Since $F$ is rationally bounded by $L$ and $F_1\sqcup F_2$ coincide with $F$ in a neighborhood of $L$, we have
 that $F_i$ has $L_i$ as boundary and then $[L_i]=[0]$ in the group $H_*(M;\Q)$. This implies that $t$ is a multiple of both $t_1$ and $t_2$.
 This proves that $t=\text{lcm}(t_1,t_2)$.
 
 We now want to prove that the Thurston norm is additive. The fact that $\norm{L_1\sqcup L_2}_T\leq\norm{L_1}_T+\norm{L_2}_T$ follows 
 immediately from Property 2 in Subsection \ref{subsection:norm}. Then we suppose that the inequality is strict: this means that
 \[-\dfrac{\chi(F_1)}{t}-\dfrac{\chi(F_2)}{t}\leq-\dfrac{\chi(F)}{t}<\norm{L_1}_T+\norm{L_2}_T\:.\] In particular, at least one 
 of the two surfaces, say $F_1$, is such that $-t^{-1}\chi(F_1)<\norm{L_1}_T$. Then the claim follows from the fact that
 by construction, if $t=at_1$, the class $[F_1]$ coincide with $a$ times the class represented by a rational Seifert surface of $L_1$ in
 $H_2(M_L,\partial M_L;\Q)$; which gives a contradiction.
 
 To conclude we need to prove that, if $L$ is the disjoint union of a link $L'$ with a compressible knot $K$, it is 
 $\norm{L}_T=\norm{L'}_T+\norm{K}_T$. This is done in the same way as the previous case, but we take into account the fact that disks
 do not increase the complexity of a surface. 
\end{proof}
We observe that the proof of this proposition also implies that the Thurston norm of a link $L$ in $M_1$ coincides with the one obtained
if $L$ is seen as a link in $M$, where $M=M_1\#M_2$. Therefore, we use the same symbol for both. Moreover, 
in order to prove our main theorem we need the following lemma.

We recall that we defined the compressibility term of a link $L$ as the rational number $\mathfrak o(L)=t_1^{-1}+...+t_{o(L)}^{-1}$, where 
$o(L)$ is the number of compressible components in $L$ and $\{t_1,...,t_{o(L)}\}$ are the orders of such components. 
\begin{lemma}
 \label{lemma:min}
 Suppose that $L$ is a non-split link in a rational homology 3-sphere $M$. Then we have that \newpage
 \[\norm{L}_T-\mathfrak o(L)=\min\left\{\dfrac{-\chi(F)}{t}\right\}\:,\]
 where $F$ is a surface in $M$ that is rationally bounded by $L$ and $t$ is the order of $L$ in $M$.
\end{lemma}
\begin{proof}
 The link $L$ can have compressible components only when it is a knot, because otherwise $L$ would be split. Thus, if $L$ is not a compressible knot then $\mathfrak o(L)=0$ and so the claim follows from the definition of
 the Thurston norm and Lemma 
 \ref{lemma:class}. On the other
 hand, if the boundary of a neighborhood of $L$ is a compressible torus in $M\setminus\mathring{\nu(L)}$ then $F$ is rationally bounded by
 $L$, where $F$ is the image of a map
 $j:D\hookrightarrow M$ as in Subsection \ref{subsection:norm} with $D$ a disk, and $\mathfrak o(L)=t^{-1}$. Therefore, 
 its Thurston norm $\norm{L}_T$ is equal to zero and the equality in the statement is given by $F$, since $\chi(F)=1$. 
\end{proof}

\section{Legendrian and transverse links in rational homology contact 3-spheres}
\label{section:three}
\subsection{Self-linking number and other classical invariants}
\label{subsection:self}
We recall that a contact structure $\xi$ on an oriented 3-manifold $M$ is a 2-plane field on $M$ such that $\xi=\text{Ker }\alpha$, where
$\alpha$ is a 1-form on $M$ and $\alpha\wedge\dd\alpha$ is a volume 3-form for $M$.
Moreover, a smooth link $T\hookrightarrow(M,\xi)$ is called transverse if $T_pT\oplus\xi_p=T_pM$ for every $p\in T$ and $\alpha\lvert_{T}$
is a volume 1-form for $T$.

Since we are working with rational homology spheres, we can define the self-linking number $\text{sl}(T)$ of a transverse link $T$ 
with order $t$ in $(M,\xi)$. 
Suppose for now that $T$ is non-split; take a surface $F$ which is rationally bounded by $T$ and the map $j:\Sigma\rightarrow M$ that defines $F$,
as described before in \ref{subsection:norm}. Consider the pull-back bundle $j^*\xi$ on $\Sigma$. Then $j^*\xi$ is trivial because it is a bundle 
over a compact, oriented surface with boundary. 
Let $v$ be a non-zero section of $j^*\xi$. Normalize $v$ so that $v\lvert_{\partial\Sigma}$ defines a link $T_{\xi}$ in $\partial\nu(T)$. We define
the
self-linking of $T$ to be \[\text{sl}(T)=\dfrac{\text{lk}_{\Q}(T,T_{\xi})}{t}=\dfrac{F\cdot T_{\xi}}{t^2}\:,\] where $F\cdot T_{\xi}$ denotes the algebraic 
intersection of the surface $F$ 
with the link $T_{\xi}$. 
This definition coincides with the one given by Baker and Etnyre in \cite{BE} for transverse knots.

In the case that $T$ has split link smooth type, we consider $T_1,...,T_k$ the non-split components of $T$ as a smooth link.
Then we say that
\[\text{sl}(T)=\sum_{i=1}^k\text{sl}(T_i)\:.\] The fact that $M$ is a rational homology sphere tells us that $\text{sl}(T)$ is independent of
the choice of the surface $F$.
Furthermore, the self-linking number is a transverse invariant; in the sense that it does
not change under transverse isotopy. We recall that two transverse links $T$ and $V$ are transverse isotopic if there is a smooth isotopy 
$G:M\times I\rightarrow M$ such that $G(T,0)=T$, $G(T,1)=V$ and $G(\cdot,t)$ is a transverse link for every $t\in I$.
\begin{prop}
 Let $T$ and $V$ be transverse isotopic transverse $n$-component links in $(M,\xi)$, where $M$ is a rational homology sphere. Then we have that
 $\emph{sl}(T)=\emph{sl}(V)$.
\end{prop}
\begin{proof}
 An easy computation gives that \[\text{sl}(U)=\sum_{i=1}^n\text{sl}(U_i)+2\cdot\sum_{i<j}\text{lk}_{\Q}(U_i,U_j)\] for every transverse link 
 $U$ with components $U_1,...,U_n$. \newpage
 Since the transverse isotopy between $T$ and $V$ sends $T_i$ into $V_i$ for every $i=1,...,n$ and it preserves the linking number between the 
 components, 
 the claim follows immediately from the previous observation and \cite{BE}, where the invariance is proved for transverse knots. 
\end{proof}
We also recall that $L$ is called Legendrian if $T_pL\subset\xi_p$ for every 
$p\in L$. Now let us consider a Legendrian link $L$ in $(M,\xi)$;
then $L$ always determines two special transverse links named positive and negative transverse push-off, that we denote with 
$T^{+}_L$ and $-T^-_L$, where the minus sign appears because transverse links need to be oriented accordingly to the contact structure.
We briefly recall the construction of $T^{\pm}_L$. Let $A_i=S^1\times[-1,1]$ be a collection of embedded annuli in $M$ such that
$S^1\times\{0\}=L_i$ and $A_i$ is transverse to $\xi$ for every $i=1,...,n$, where $n$ is the number of components of $L$.
If the $A_i$'s are sufficiently ``thin'', in the sense of Section 2.9 in \cite{Etnyre}, then $T^{\pm}_L$ is the link with components 
$S^1\times\{\pm\frac{1}{2}\}$,
which is transverse up to orientation. It is easy to check that any two positive (or negative) transverse push-offs are transversely isotopic and 
then $T^{\pm}_L$
is uniquely defined, up to transverse isotopy. Moreover, if we reverse the orientation of $L$ then we have $T^\pm_{-L}=-T^\mp_L$.

Using the transverse push-offs we can easily define the Thurston-Bennequin and rotation numbers of a Legendrian link $L$. Namely, we say that
\[\text{tb}(L)=\dfrac{\text{sl}\left(T^+_L\right)+\text{sl}\left(-T^-_L\right)}{2}\:\:\:\:\:\text{ and }\:\:\:\:\:
\text{rot}(L)=\dfrac{\text{sl}\left(-T^-_L\right)-\text{sl}\left(T^+_L\right)}{2}\:.\] It is clear that $\text{tb}(L)$ and $\text{rot}(L)$ are 
Legendrian isotopy invariants. We recall as before that $L_1$ and $L_2$ are Legendrian isotopic if there is a smooth isotopy 
$G:M\times I\rightarrow M$ such that $G(L_1,0)=L_1$, $G(L_1,1)=L_2$ and $G(\cdot,t)$ is a Legendrian link for every $t\in I$.

\subsection{The Thurston-Bennequin inequality}
A contact structure $\xi$ on $M$ is called overtwisted if there exists an embedded disk $D$ in $M$ such that $\partial D$ is Legendrian and
$\text{tb}(\partial D)=0$; such disk is called overtwisted disk. On the other hand, a contact 3-manifold 
$(M,\xi)$ is tight if it is not overtwisted.
Then the proof of Theorem \ref{teo:main} follows from the following lemma; this strategy is the same that appears in \cite{BE,Eliashberg}.
\begin{lemma}
 \label{lemma:char}
 Suppose that $T$ is a non-split transverse link in a rational homology tight 3-sphere $(M,\xi)$ and take an $F$ in $M$
 that is rationally bounded by $T$. Then we can perturb $F$ in a way that $T'=F\cap\partial\nu(T)$ is still transverse
 and $\self(T')=t\cdot\self(T)$.
 
 Furthermore, we have that \[F\cdot T'_{\xi}\leq-\chi(F)\:,\] where the framing $T'_{\xi}$ is defined before 
 in Subsection \ref{subsection:self}.
\end{lemma}
\begin{proof}
 Consider the map $j:\Sigma\rightarrow M$ which determines $F$. Let us take a small neighborhood $\nu(T)$ of $T$ such that it
 intersects $F$ only in the image of a neighborhood of $\partial\Sigma$, as in the proof of Proposition \ref{prop:split}.
 Denote the properly embedded surface $M_L\cap F$ with $F'$, where the manifold $M_L$ is $M\setminus\mathring{\nu(T)}$.
 
 We know from \cite{BE} that we can modify $F'$ in a way
 that $\partial F'=T'$ is as wanted and 
 the characteristic foliation $F'_{\xi}$ is generic.
 In particular, 
 we can assume that all the singularities are isolated elliptic or hyperbolic points. Moreover, each singularity has a sign
 depending on whether the orientation of $\xi$ and $TF'$ agree at the singularity. Let us denote with $e_{\pm}$ and $h_{\pm}$ 
 the number of such singularities.
 
 Since we can interpret $F'\cdot T_{\xi}$ as a relative Euler class, in other words $F'\cdot T_{\xi}$ is the obstruction
 to extending the framing $T_{\xi}$ to a non-zero vector field on $F'$, we have that
 \[F\cdot T'_{\xi}=F'\cdot T'_{\xi}=(e_--h_-)-(e_+-h_+)\:,\]
 as in \cite{BE}.
 Moreover, a simple computation gives that \[\chi(F)=(e_++e_-)-(h_++h_-)\:;\] thus it is \newpage
 \begin{equation}
  \label{inequality}
  F\cdot T'_{\xi}+\chi(F)=2(e_--h_-)\:.
 \end{equation} 
 At this point, since $\xi$ is tight, every negative elliptic point is connected to a
 negative hyperbolic point, otherwise we could find overtwisted disks in $(M,\xi)$, and then we can cancel this pair using 
 Giroux's elimination lemma \cite{Giroux}.
 This means that we can isotope $F'$ so that the characteristic foliation is such that $e_-=0$ and then the claim follows easily 
 from Equation \eqref{inequality}.
\end{proof}
We can now prove the main result of the paper.
\begin{proof}[Proof of Theorem \ref{teo:main}]
 Suppose first that $T$ is a non-split transverse link of order $t$ in $(M,\xi)$. Take an $F$ that is rationally bounded by $T$ 
 and gives the equality in Lemma \ref{lemma:min}. 
 Then from Lemma \ref{lemma:char} we have that
 \[\self(T)=\dfrac{\self(T')}{t}=\dfrac{F\cdot T'_{\xi}}{t}\leq-\dfrac{\chi(F)}{t}=\norm{T}_T-\mathfrak o(T)\:.\]  
 For the second equality we also use that $T'$ is null-homologous.
 
 If $T$ has split smooth link type then we reason as in Subsection \ref{subsection:self} and we write $T_1,...,T_k$, which are
 the split 
 components of $T$. Hence, this time we obtain
 \[\text{sl}(T)=\sum_{i=1}^k\text{sl}(T_i)\leq\sum_{i=1}^k\norm{T_i}_T-\mathfrak o(T_i)=\norm{T}_T-\mathfrak o(T)\:,\] because of Proposition
 \ref{prop:split} and the additivity of the Thurston norm and the compressibility term.
 
 Now suppose that $L\hookrightarrow(M,\xi)$ is a Legendrian link. Then from the definition of Thurston-Bennequin and rotation numbers we have that
 \[\text{tb}(L)\mp\text{rot}(L)=\text{sl}\left(\pm T^\pm_L\right)\leq\norm{L}_T-\mathfrak o(L)\:;\] in fact, $L$ and its transverse push-offs are 
 all smoothly 
 isotopic, up to orientation, and then they have same Thurston norm and number and order of compressible components. 
\end{proof}

\section{Quasi-positive links}
\label{section:four}
\subsection{Maximal self-linking number and the slice genus}
\label{subsection:max_self}
A \emph{quasi-positive link} is any link which can be realized as the closure 
of a $d$-braid of the
form \[\prod_{i=1}^bw_i\sigma_{j_i}w_i^{-1}\:,\] where $\sigma_j$ for $j=1,...,d-1$ are the generators of the $d$-braids group.
Thus quasi-positive links are closures of braids consisting of arbitrary conjugates of
positive generators. 

It is worth noting that quasi-positive links are equivalent to another, more geometric
class of links: the \emph{transverse $\C$-links}. These links arise as the transverse intersection of 
the 3-sphere $S^3\subset\C^2$, with the complex curve $f^{-1}(0)$, where $f:\C^2\rightarrow\C$ is a non-constant polynomial. Transverse $\C$-links include 
links of isolated curve singularities, but are in fact a much larger class. The fact that quasi-positive
links can be realized as transverse $\C$-links is due to Rudolph \cite{Rudolph2}, while the fact that
every transverse $\C$-link is quasi-positive is due to Boileau and Orekov \cite{Boileau}.

Given a quasi-positive braid $B=(w_1\sigma_{j_1}w_1^{-1})\cdot...\cdot(w_b\sigma_{j_b}w_b^{-1})$, 
we can associate to $B$ a surface $\Sigma_B$ as follows. Let us consider the braid $B'$, obtained
by removing the $\sigma_{j_i}$'s from the presentation of $B$. Then $B'$ is the boundary of $d$ disks
with some ribbon interesction between themselves. If we push these disks in the 4-ball then the intersections disappear and
we obtain a surface which is properly embedded in $D^4$. At this point, we add $b$ negative bands in correspondence of the $\sigma_{j_i}$'s;
the result is an oriented and 
compact surface that is not emebedded in $S^3$, but it is properly embedded in $D^4$, whose boundary is clearly the closure of the 
braid $B$.

A \emph{connected transverse $\C$-link} is a link which is the closure of a quasi-positive braid $B$ such that $\Sigma_B$ is 
connected. Then we have the following proposition.
\begin{prop}
 A link $L$ is a connected transverse $\C$-link if and only if there exists a non-constant, 2-variable, complex polynomial $f$ such that
 $L=S^3\pitchfork f^{-1}(0)$ and $\Sigma_L=D^4\pitchfork f^{-1}(0)$ is connected.
\end{prop}
\begin{proof}
 If $L$ is a connected transverse $\C$-link, represented by a quasi-positive braid $B$, then Rudolph showed \cite{Rudolph2} that $\Sigma_B$ is precisely 
 the intersection between the 4-ball in $\C^2$ with a complex curve, which can be seen as the zero locus of a non-constant polynomial.
 
 Conversely, Boileau and Orevkov in \cite{Boileau} proved that, given $L$ and $\Sigma_L=D^4\pitchfork f^{-1}(0)$ as above, there is an isotopy
 $H:S^3\times I\rightarrow S^3$ such that $H(L,0)=L$ and $H(L,1)=L'$, where $L'$ is the closure of a quasi-positive braid $B'$.
 Moreover, the surface \[\Sigma_{L'}=\Sigma_L\cup\left(\bigcup_{t\in I}H(L,t)\right)\] is still properly embedded in the 4-ball obtained by
 gluing $S^3\times I$ to $D^4$ along a boundary $S^3$, and coincide with $\Sigma_{B'}$. Clearly, the surface $\Sigma_{L'}$ is still connected and 
 then $L'=L$ is a connected transverse $\C$-link.  
\end{proof}
Let us consider a quasi-positive $d$-braid $B$ of length $b$ as before. Then we call $\mathcal B$ the graph with $d$ vertexes, corresponding to the 
$d$ strings in $B$, such that an edge between the $i$-th and $j$-th vertices appears in the case that, for some $1\leq k\leq b$ and $w$, one has
\[B=(w_1\sigma_{j_1}w_1^{-1})\cdot...\cdot(w\sigma_kw^{-1})\cdot...\cdot(w_b\sigma_{j_b}w_b^{-1})\] 
and $\sigma_k$ corresponds to a negative band connecting the $i$-th and the $j$-th strings of $B$. Then it follows from the construction
that the connected components of $\Sigma_B$ correspond to the components of $\mathcal B$; in particular, the surface $\Sigma_B$ is connected if and only 
if the graph $\mathcal B$ is connected.
\begin{prop}
 Every connected transverse $\C$-link is non-split quasi-positive.
\end{prop}
\begin{proof}
 Suppose that $B$ is a quasi-positive braid, presented as before, such that $\Sigma_B$ is connected and its closure $L$ is a split link. We can write
 $L=L_1\sqcup L_2$ and $L_i$ is the closure of the subbraid $B_i$ for $i=1,2$. Then,
 it follows from \cite{Orevkov} that $\Sigma_B=\Sigma_{B_1}\sqcup\Sigma_{B_2}$ and this is a contradiction.
\end{proof}
Later, in Corollary \ref{cor:strict} we show that this inclusion is strict.

In \cite{Cavallo} we proved that the Thurston-Bennequin inequality in the standard 3-sphere 
can be improved using the invariant $\tau(L)$, defined from the filtered link Floer homology group $\widehat{\mathcal{HFL}}(L)$.
More specifically, we showed that the following relation holds:
\begin{equation}
 \label{tb_bound}
 \tb(\mathcal L)+|\rot(\mathcal L)|\leq 2\tau(L)-n\:,
\end{equation} 
where $\mathcal L$ is an $n$-component Legendrian link in $(S^3,\xist)$. From this result we obtain the following corollary.
\begin{cor}
 Suppose that $\mathcal L$ is an $n$-component Legendrian link in $(S^3,\xi_{\text{st}})$ with smooth link type $L$. Then we have that
 \begin{equation}
  \label{slice}
  \self(\mathcal T_{\mathcal L})\leq\tb(\mathcal L)+|\rot(\mathcal L)|\leq2\tau(L)-n\leq-\chi(\Sigma)\:,
 \end{equation}
 where $\Sigma$ is an 
 oriented, compact surface, properly embedded in $D^4$, such that $\partial\Sigma=\mathcal L$.
\end{cor} 
\begin{proof}
 The first two inequalities follow from the properties of the transverse push-off and Equation \eqref{tb_bound}. The last one is 
 a consequence of Proposition 4.7 in \cite{Cavallo}.
\end{proof}
We now prove Theorem \ref{teo:self}.
\begin{proof}[Proof of Theorem \ref{teo:self}]
 The braid $B=(w_1\sigma_{j_1}w_1^{-1})\cdot...\cdot(w_b\sigma_{j_b}w_b^{-1})$ determines a 
 transverse link $\mathcal T$ in $(S^3,\xist)$ and \[\self(\mathcal T)=\writhe(B)-d=b-d\] where $\writhe(B)$ is the writhe of $B$, see \cite{Etnyre}.
 Then, since transverse links are in bijection with Legendrian links up to negative stabilization \cite{Etnyre},
 Equation \eqref{slice} implies that 
 \[\self(\mathcal T)\leq\text{SL}(L)\leq2\tau(L)-n\leq-\chi(\Sigma_B)=b-d\] and these inequalities are all equalities.
\end{proof}
When we consider connected surfaces, the inequality in Equation \eqref{slice} becomes
\begin{equation}
 \label{slice2}
 \self({\mathcal T})\leq2\tau(L)-n\leq2g_4(L)+n-2\:, 
\end{equation}
for every $n$-component transverse link $\mathcal T$ in $(S^3,\xist)$, with link type $L$, and 
where we recall that $g_4(L)$ is the minimal genus of a connected, compact, oriented surface $\Sigma$ properly embedded in $D^4$ and such that 
$\partial\Sigma=\mathcal T$.
Then we have the following proposition.
\begin{prop}
 \label{prop:slice}
 Suppose that $L$ is an $n$-component link, embedded in $S^3$, which is a connected transverse $\C$-link. Then we have that
 $\tau(L)=g_4(L)+n-1$.
\end{prop}
\begin{proof}
 The claim follows from the same argument in the proof of Theorem \ref{teo:self}: let $\Sigma'$ be a surface that attains $g_4(L)$, and note that $\Sigma'$ may differ from $\Sigma_B$; Equation \eqref{slice2} gives \[2\tau(L)-n\leq-\chi(\Sigma')=2g(\Sigma')+n-2=2g_4(L)+n-2\] and then the equality follows from \[2g_4(L)+n-2\leq2g(\Sigma_B)+n-2=-\chi(\Sigma_B)=2\tau(L)-n\:.\] 
\end{proof}
This is a generalization of a result of Plamenevskaya on quasi-positive knots \cite{Plamenevskaya2},
which of course are connected transverse $\C$-links. 

An implication of this proposition is the additivity of the slice genus under connected sums that we stated in Corollary \ref{cor:slice}.
\begin{proof}[Proof of Corollary \ref{cor:slice}]
 If $L_1$ and $L_2$ are two connected transverse $\C$-links then it is easy to see that every connected sum $L=L_1\#\:L_2$ has the same
 property. In fact, take two quasi-positive braids, representing $L_1$ and $L_2$; then $L$ is obtained by putting the second below the 
 first one and adding a negative band between the components that we want to sum, say the $i$-th and the $j$-th ones with $i<j$. 
 This move can be seen as the composition of $B$ with $w\sigma_{j-1}w^{-1}$ for some $w$ and then
 the resulting braid is still quasi-positive.
 
 Therefore, Proposition \ref{prop:slice} gives that $\tau(L_i)=g_4(L_i)+n_i-1$, where $n_i$ are the number of components of $L_i$ for $i=1,2$.
 Moreover, what we said before also implies that $\tau(L)=g_4(L)+(n_1+n_2-1)-1$. At this point, we use 
 that the $\tau$-invariant is additive under connected sums of links, see Equation (11) in \cite{Cavallo}, and then the proof is complete.
\end{proof}

\subsection{Proof of Theorem \ref{teo:strong}}
\label{subsection:quasipositive}
We apply the results obtained in the previous section to recompute the invariant tau for strongly quasi-positive links in $S^3$. 
A link $L\hookrightarrow S^3$ is \emph{strongly quasi-positive} if it is the boundary of a quasi-positive surface $F$. Such surfaces are 
constructed in the following way: take $d$ disjoint parallel, embedded disks, all of them oriented in the same way, and attach $b$ negative
bands on them, each one between a pair of distinct disks. This procedure is shown in Figure \ref{QP}. 
The negative bands cannot be knotted with each other; this means that, up to isotopy, $F$ only depends on the number $d$ and the ordered $b$-tuple
$\textbf b=(\sigma_{i_1j_1},...,\sigma_{i_bj_b})$, where $\sigma_{ij}$ denotes that a negative band is put between the $i$-th and the $j$-th
disk with $i<j$. 

Denote with $\sigma_1,...,\sigma_{d-1}$ the generators of the $d$-braids group. Then it is easy to see that the boundary of the quasi-positive
surface $(1,\sigma_{ij})$ is isotopic to the closure
of the $d$-braid given by 
\[\left\{\begin{aligned}
               (\sigma_i\cdot\cdot\cdot\sigma_{j-2})\sigma_{j-1}(\sigma_i\cdot\cdot\cdot\sigma_{j-2})^{-1}\:\:\:\:\:\text{if }j\geq i+2& \\
               \sigma_i\:\:\:\:\:\hspace{2.7cm}\text{if }j=i+1\\
                     \end{aligned}\right.\:;\] see \cite{Hedden} for more details. This immediately implies that strongly quasi-positive links
are quasi-positive.

Suppose $S^3$ is equipped with its unique tight contact structure $\xi_{\text{st}}$
and $\mathcal L\hookrightarrow(S^3,\xi_{\text{st}})$ is an 
$n$-component transverse link with smooth link type $L$. Then we have the following corollary.
\begin{cor}
 \label{cor:inequality}
 Suppose that $\mathcal T$ is an $n$-component transverse link in $(S^3,\xi_{\text{st}})$ with smooth link type $L$. Then we have that
 \[\self(\mathcal T)\leq 2\tau(L)-n\leq\norm{L}_T-o(L)\:,\] where $o(L)$ is the number of 
 disjoint unknotted unknots in $L$.
\end{cor}
\begin{proof}
 The first inequality follows immediately from Equation \eqref{slice2}. 
 For the second one we 
 apply Theorem \ref{teo:detect}; in fact, it is
 known, see \cite{Book}, that the group $\widehat{HFL}_{*,*}(L)$ is non-zero in bigrading $(0,\tau(L))$. 
\end{proof}
We can now prove Theorem \ref{teo:strong}.
\begin{proof}[Proof of Theorem \ref{teo:strong}]
 For Lemma \ref{lemma:min} and Corollary \ref{cor:inequality} 
 it is enough to show that, given a quasi-positive surface $F$ whose boundary is $L$, we can find a transverse link $\mathcal T$, with smooth link 
 type $L$,
 such that $\self(\mathcal T)=b-d$.
 
 Let us start with the $d$ parallel disks in $F$ and see what happens when we attach the first negative band. 
 \begin{figure}[ht]
  \centering
  \psfig{file=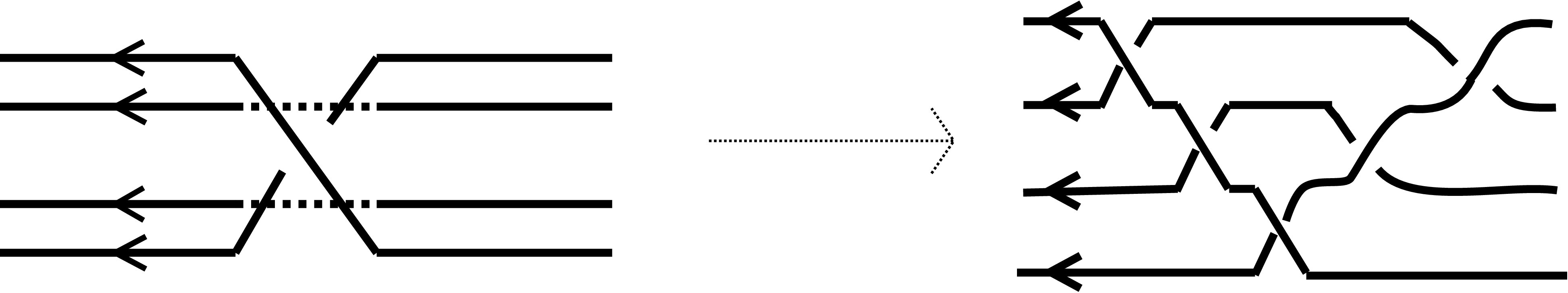,width=9cm}   
  \caption{Transverse realization of $L_1$.}
  \label{Leg}
 \end{figure}
 The surface we obtain is the 
 quasi-positive surface determined by $d$ and $\textbf b=(\sigma_{ij})$, where $i<j$ correspond to the disks on which we glue the negative band,
 and we call $L_1$ its boundary.
                      
 At this point, we choose a transverse representative of $L_1$ as shown in Figure \ref{Leg}.
 Then $\mathcal T$ is defined by iterating this
 procedure with all the $b$ bands and, since we know how to compute its self-linking number \cite{Etnyre}, we obtain that
 \[\self(\mathcal T)=b-d\:.\] This proves the claim.
\end{proof}
Since we remarked in the proof of Corollary \ref{cor:inequality}
that $\widehat{HFL}_{0,\tau(L)}(L)\neq\{0\}$, then Theorems \ref{teo:detect} and \ref{teo:strong} tells us that, for strongly
quasi-positive links, one has \newpage \[\tau(L)=\max\left\{s\in\Z\:|\:\widehat{HFL}_{*,s}(L)\neq\{0\}\right\}\:.\]
Furthermore, applying Lemma \ref{lemma:min} to Theorem \ref{teo:strong},
we can say more when the link $L$ also bounds a connected quasi-positive surface, that we call \emph{quasi-positive
Seifert surface} for $L$.
\begin{cor}
 \label{cor:strong}
 If $L$ is a link as in Theorem \ref{teo:strong}, which also admits a quasi-positive Seifert surface $F'$, then we have that 
 \[\tau(L)=g_3(L)+n-1=g(F')+n-1\:,\] where $g_3(L)$ is the Seifert genus of $L$.
\end{cor} 
This corollary implies that all the quasi-positive Seifert surfaces of a given link have the same genus.
Moreover, we observe that not every strongly quasi-positive link admits a quasi-positive Seifert surface. In fact, the 4-component link $L$ in
Figure \ref{QP} is such that $\tau(L)=2$ from Theorem \ref{teo:strong}, but $L$ has Seifert genus zero: to see this it is enough to add a
tube between disks 1 and 2 in Figure \ref{QP}. Then
we have that
\[2=\tau(L)\neq g_3(L)+n-1=3\] and so this would contradict Corollary \ref{cor:strong}.
This argument also implies the following corollary.
\begin{cor}
 \label{cor:strict}
 The link $L$ in Figure \emph{\ref{QP}} is non-split quasi-positive, but it is not a connected transverse $\C$-link.
\end{cor}
\begin{proof}
 By construction $L$ is non-split strongly quasi-positive and then it is quasi-positive. If we suppose that $L$ is a connected transverse 
 $\C$-link then we can use Proposition \ref{prop:slice} and we obtain that \[g_4(L)=\tau(L)+1-n=-1\] which is clearly impossible.
\end{proof}

\subsection{The case of positive links}
Given a diagram $D$ for an oriented link $L$, we define the \emph{oriented resolution} of $D$ as the collection of circles in $\R^2$ obtained by
resolving all the crossings in $D$ preserving the orientation.
\begin{figure}[ht]
 \centering
 \psfig{file=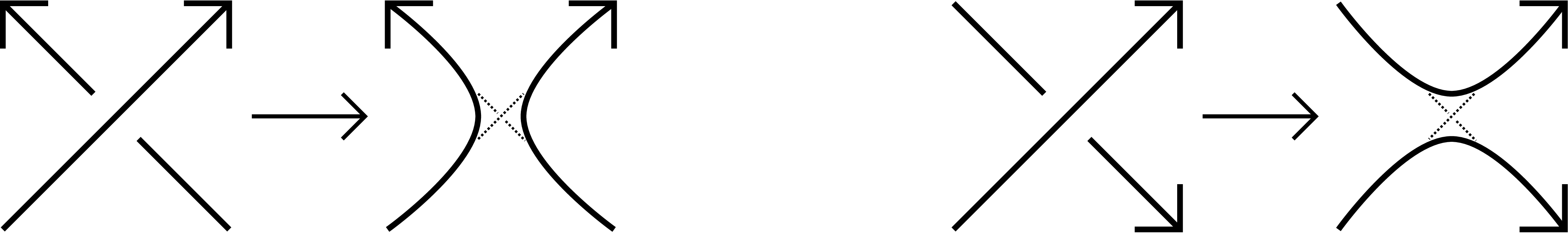,width=11cm}
 \caption{Examples of orientation preserving resolutions.}
 \label{Resolutions}
\end{figure}
It is easy to see that, for each crossing, this
can be done only in one possible way, as shown in Figure \ref{Resolutions}. The circles have the orientation induced on them by $D$. 

From the oriented resolution we can construct a compact, oriented surface $F$ in the 3-space which is bounded by the link $L$. 
We start by pushing all
the circles up, starting from the innermost ones, until they are all on different levels and then we take the disks that they bound. 
So we now have a collection of disjoint disks in $\R^3$. We connect these disks by attaching negative (positive) bands in correspondence of the
positive (negative) crossings in $D$. 

The orientation on $F$ is defined in the following way: it coincides with the one induced by $\R^2$ on the 
circles oriented counter-clockwise, while it is the opposite on the circles oriented clockwise. It is easy to check that such orientation can be 
extended to the whole surface $F$. Furthermore, the surface $F$ is connected if and only if the diagram $D$ is non-split.

We recall that, from \cite{Nakamura,Rudolph},
positive links are always strongly quasi-positive. Moreover, if a link is positive non-split then it also
admits a quasi-positive Seifert surface. This means that the results in Subsections \ref{subsection:max_self} and \ref{subsection:quasipositive} hold in this case too.
Moreover, we state a proposition from \cite{Kalman}, which tells us that the maximal Thurston-Bennequin number $\text{TB}$ of a positive 
link can be determined from oriented resolutions.
\begin{teo}[K\'alm\'an]
 \label{teo:positive}
 Suppose that $L$ is an $n$-component link in $S^3$ with a positive diagram $D$. Then we have that 
 \[\emph{TB}(L)=c(D)-k(D)\:,\] where $c(D)$ and $k(D)$ are the numbers of crossings in $D$ and of circles
 in the oriented resolution of $D$.
\end{teo}
Then we have the following result.
\begin{proof}[Proof of Proposition \ref{prop:positive}]
 We suppose first that $L$ is a non-split link. Then from Corollary \ref{cor:inequality} we know that 
 \begin{equation}
  \label{positive}
  \text{TB}(L)\leq 2\tau(L)-n=2g_3(L)+n-2\leq-\chi(F)=c(D)-k(D)\:;
 \end{equation}
 where $F$ is the surface obtained from the oriented resolution of $D$, which is a Seifert surface for $L$ \cite{Rudolph}. The first
 equality follows from the fact that non-split positive links admit a quasi-positive Seifert surface and Corollary \ref{cor:strong}.
 Hence, from Theorem \ref{teo:positive} all the inequalities in 
 Equation \eqref{positive} are equalities. 
 Suppose now that $L=L_1\sqcup...\sqcup L_r$, where the $L_i$'s are the split components of $L$. 
 
 Take the Legendrian links $\mathcal L_i$, representing $L_i$ for $i=1,...,r$, which satisfy
 the equality before and denote with $\mathcal L$ the Legendrian link $\mathcal L_1\sqcup...\sqcup\mathcal L_r$. Then we have that
 \begin{equation}
  \label{sharp}
  \tb(\mathcal L)=\sum_{i=1}^r\tb(\mathcal L_i)=\sum_{i=1}^rc(D_i)-k(D_i)=c(D)-k(D)\:,
 \end{equation}
 where $D_i\subset D$ is the positive subdiagram representing $L_i$. This holds because it is easy to check that the split components of a
 positive diagram are still positive; moreover, a positive diagram is non-split if and only if it represents a non-split link.
 
 Then the Thurston-Bennequin inequality in Corollary \ref{cor:strong} again gives
 \begin{equation}
  \label{positive2}
  \begin{aligned}
              \tb(\mathcal L)\leq\text{TB}(L)\leq 2\tau(L)-n=\norm{L}_T-o(L)\leq&\sum_{i=1}^r\left(2g_3(L_i)+n_i-1\right)\leq \\
              \leq-\sum_{i=1}^r\chi(F_i)=c(D)-k(D)\:, \\
                     \end{aligned}  
 \end{equation}
 where $F_i$ are the Seifert surfaces obtained from the oriented resolution of $D_i$ for $i=1,...,r$, the number $o(L)$ tells us how many disjoint unknotted 
 unknots there are in $L$ and $n_i$ is the number of components of each $L_i$. 
 We also used Proposition \ref{prop:split}.
 Clearly, Equation \eqref{sharp} says that all the inequalities in
 Equation \eqref{positive2} are equalities and then the claim follows from the fact that
 \[\sum_{i=1}^r(2g_3(L_i)+n_i-1)=2\sum_{i=1}^rg_3(L_i)+n-r=2g_3(L)+n-r\:,\] where the final equality holds because the Seifert genus is additive
 under disjoint unions.
\end{proof}
The invariant $\tau(L)$ of a positive link can be determined from the oriented resolution of a positive diagram $D$ for $L$. In fact, 
Proposition \ref{prop:positive} and its proof immediately imply the following corollary.
\begin{cor}
 \label{cor:positive}
 Suppose that $L$ is a positive $n$-component link in $S^3$ with $r$ split components. Then we have that \newpage
 \[\tau(L)=\sum_{i=1}^rg(F_i)+n-r=\dfrac{c(D)-k(D)+n}{2}\:,\] where $\{F_i\}_{i=1}^r$ are the connected components of the surface obtained
 from the oriented resolution of a positive diagram $D$ for $L$, while $c(D)$ and $k(D)$ are as in Proposition \ref{prop:positive}.
\end{cor}
In particular, all the surfaces obtained from the oriented resolution of a positive diagram for a given link have the same genus. 
Furthermore, Corollary \ref{cor:positive} also gives that a positive link admits a quasi-positive Seifert surface if and only if it is non-split.


\begin{thebibliography}{9}  
 \bibitem{BE} K. Baker and J. Etnyre, \emph{Rational linking and contact geometry},
             Progr. Math., 296, Birkh\"auser/Springer, New York, 2012. 
 \bibitem{Boileau} M. Boileau and S. Y. Orevkov, \emph{Quasi-positivit\'e d'une courbe analytique dans une boule pseudo-convexe},
             C. R. Acad. Sci. Paris., \textbf{332} (2001), pp. 825--830.            
 \bibitem{Calegari} D. Calegari, \emph{Foliations and the geometry of 3-manifolds},
             Oxford Mathematical Monographs, Oxford University Press, Oxford, 2007.            
 \bibitem{Cavallo} A. Cavallo, \emph{The concordance invariant tau in link grid homology},    
                Algebr. Geom. Topol., \textbf{18} (2018), no. 4, pp. 1917--1951.
 \bibitem{DM} O. Dasbach and B. Mangum, \emph{On McMullen's and other inequalities for the Thurston norm of link complements},          
             Algebr. Geom. Topol., \textbf{1} (2001), pp. 321--347.
 \bibitem{Eliashberg} Y. Eliashberg, \emph{Legendrian and transversal knots in tight contact 3-manifolds},            
             Topological methods in modern mathematics, pp. 171-193, Publish or Perish, Houston, TX, 1993.           
 \bibitem{Etnyre} J. Etnyre, \emph{Legendrian and transversal knots. Handbook of knot theory},
             Elsevier B. V., Amsterdam, 2005, pp. 105--185.
 \bibitem{Giroux} E. Giroux, \emph{Structures de contact en dimension trois et bifurcations des feuilletages de surfaces},
             Invent. Math., \textbf{141} (2000), no. 3, pp. 615--689.
 \bibitem{Hayden} K. Hayden, \emph{Quasipositive links and Stein surfaces},             
             arXiv:1703.10150.            
 \bibitem{Hedden} M. Hedden, \emph{Notions of positivity and the Ozsv\'ath-Szab\'o concordance invariant},
             J. Knot Theory Ramifications, \textbf{19} (2010), no. 5, pp. 617--629.
 \bibitem{Kalman} T. K\'alm\'an, \emph{Maximal Thurston-Bennequin number of $+$adequate links},
             Proc. Amer. Math. Soc., \textbf{136} (2008), no. 8, pp. 2969--2977.        
 \bibitem{Nakamura} T. Nakamura, \emph{Four-genus and unknotting number of positive knots and links},
             Osaka J. Math., \textbf{37} (2000),  pp. 441--451.            
 \bibitem{Ni} Y. Ni, \emph{Link Floer homology detects the Thurston norm},
             Geom. Topol., \textbf{13} (2009), no. 5, pp. 2991--3019.    
 \bibitem{Orevkov} S. Y. Orevkov, \emph{Quasipositive braids and connected sums}, 
             arXiv:1906.03454.
 \bibitem{Book} P. Ozsv\'ath, A. Stipsicz and Z. Szab\'o, \emph{Grid homology for knots and links},
             AMS, volume 208 of Mathematical Surveys and Monographs, 2015.            
 \bibitem{OS} P. Ozsv\'ath and Z. Szab\'o, \emph{Link Floer homology and the Thurston norm},
             J. Amer. Math. Soc., \textbf{21} (2008), no. 3, pp. 671--709. 
 \bibitem{Plamenevskaya2} O. Plamenevskaya, \emph{Transverse knots and Khovanov homology},
             Math. Res. Lett., \textbf{13} (2006), no. 4, pp. 571--586.            
 \bibitem{Rudolph2} L. Rudolph, \emph{Algebraic functions and closed braids},  
             Topology, \textbf{22} 1983, no. 2, pp. 191--202.            
 \bibitem{Rudolph} L. Rudolph, \emph{Positive links are strongly quasipositive},   
             Geom. Topol. Monogr. (2), Geom. Topol. Publ., Coventry, 1999. 
 \bibitem{Thurston} W. Thurston, \emph{A norm for the homology of 3-manifolds},
             Mem. Amer. Math. Soc., \textbf{59} (1986), no. 339, pp. i--vi and 99--130. \newpage
\end{thebibliography}
\end{document}